\documentclass[12pt,reqno]{amsart}

\usepackage{amsmath,amssymb,amsthm}

\usepackage{graphicx}

\usepackage{subcaption} %para facer a artistada coas figuras nun grid
\usepackage{adjustbox}
\usepackage{tikz-cd}
\usepackage{hyperref}
\hypersetup{
	unicode=true,
	colorlinks=true,
	citecolor=black,
	linkcolor=black,
	urlcolor=black,
	plainpages=false,
 }

 \usepackage{url}	%url adresses
 \allowdisplaybreaks %align may cross pages

\usepackage{cancel}
\usepackage{pgf}

\usepackage{xcolor}

\usepackage{comment} % para comentar trozos grandes de codigo

\newtheorem{teo}{Theorem}[section]
\newtheorem{lemma}[teo]{Lemma}

\theoremstyle{definition}
\newtheorem{defi}[teo]{Definition}
\newtheorem{example}[teo]{Example}
\newtheorem{remark}[teo]{Remark}

\theoremstyle{remark}

\numberwithin{figure}{section}%figure numbering

\begin{document}

\title[Morse-Bott inequalities for endomorphisms]{Morse-Bott inequalities for endomorphisms}

\author[E. Mac\'ias]{%
	Enrique Mac\'ias-Virg\'os %etc.
}

 \address{%
           	E. Macías-Virgós\\
              CITMAga, Departamento de Matem\'aticas, Universidade de Santiago de Compostela, SPAIN}
               \email{quique.macias@usc.es}

\author[A.\ O. Majadas-Moure]{%
	Alejandro O. Majadas-Moure %etc.
}

\address{
		 Alejandro O. Majadas-Moure \\
		 Departamento de Matemáticas, Universidade de Santiago de Compostela, SPAIN}
		\email{alejandro.majadas@usc.es}

\author[D. Mosquera]{%
	David Mosquera-Lois %etc.
}

 \address{%
           	David Mosquera-Lois \\
              Departamento de Matemáticas, Universidade de Vigo, SPAIN}
               \email{david.mosquera.lois@uvigo.gal} 

\author[J.\ A. Vilches]{José Antonio Vilches}
\address{%
           	José Antonio Vilches \\
              Departamento de Geometría y Topología, Universidad de Sevilla, SPAIN}
               \email{vilches@us.es}


\begin{abstract} Let $K$ be a finite simplicial complex, let $g\colon K\to K$ be a simplicial map and let $f$ be a discrete Morse-Bott function on $K$ satisfying $f(g(\sigma))\leq f(\sigma)$ for all simplices $\sigma$ in $K$. We establish a set of inequalities (generalizing the Morse-Bott inequalities which we recover as a particular case when $g$ is the identity) relating the dynamics of $g$ and $f$.
\end{abstract}

\keywords{Discrete Morse-Bott functions, Lefschetz number}


\subjclass{57Q70}

%\pacs[MSC Classification]{35A01, 65L10, 65L12, 65L20, 65L70}

\thanks{The authors were partially supported by MICINN-Spain research project PID2020-114474GB-I00. The last author was additionally partially supported by MICINN-Spain research project PID2023-149804NB-I00 and by Junta de Andalucía research project PROYEXCEL-00827. The authors thank the anonymous referees for their careful reading of the manuscript and for their comments.}

\maketitle

%\tableofcontents

\section*{Introduction}\label{sec:intro}
 
Both Morse theory and the Lefschetz number establish connections between Topology and Dynamics. Morse theory, when applicable, offers a more refined relationship and provides additional information about critical points. However, the Lefschetz number applies in a broader range of settings. E. Pitcher (see \cite{Pitcher}) posed the question of extending the Lefschetz fixed point formula to provide additional information or broadening the applicability of Morse theory. He stated a result \cite[Theorem 2]{Pitcher} addressing this issue in the setting of smooth manifolds and Morse dynamics. In this short paper we further investigate this question and extend Pitcher's statement in several directions which we elaborate on after fixing some notation. Let $K$ be a finite simplicial complex, let $g\colon K \to K$ be a simplicial map and let $f\colon K\to \{0,1,\ldots, m\}$ be a Forman-Morse-Bott function (see Definition \ref{defi:Forman-Morse-Bott}) satisfying $f(g(\sigma))\leq f(\sigma)$ for every simplex $\sigma$. We refer the reader to \cite{Forman,Knudson,Scoville} for discrete Morse theory, to \cite{Forman_Bott} for the foundational paper on discrete Morse-Bott theory on simplicial complexes, to \cite{SQDJ} for discrete Morse-Bott theory on posets, and to \cite{DML} for a detailed treatment of discrete Morse-Bott theory in both contexts.

Our result (Theorem \ref{thm:teo}) compares to Pitcher's statement as follows. First, we work in the setting of discrete Morse-Bott theory on simplicial complexes. Therefore, we state and prove a result not only for discrete Morse functions but also for Morse-Bott dynamics, dealing with general combinatorial vector fields (not restricted to gradient vector fields).  Second, our statement involves not only the map $g\colon K \to K$, but also its powers $g^l\colon K \to K$ (composing it with itself $l$ times for $l\geq 1$), providing, in some sense, which we elaborate on below, a measure of the persistence of $g$. Third, we provide full proofs and a fully worked-out example.

These results have natural applications in discrete dynamical systems where the simplicial map $g$ models a timestep evolution. When $g$ is compatible with a Forman-Morse-Bott function $f$, the inequalities provide topological constraints on the dynamics, quantifying the persistence of features across iterations and offering a localized description of system stability from the viewpoint of Morse-Bott theory.

%\subsection*{Funding} REVISAR The authors were partially supported by  Xunta de Galicia ED431C 2023/31 with FEDER funds and by MICINN-Spain research project PID2020-114474GB-I00.

%\subsection*{Acknowledgements} The author thanks Nicholas A. Scoville and Kevin P. Knudson for enlightening discussions regarding the topic of the paper.

\section{The result}

\subsection{Notation and a technical lemma} We begin by introducing some notation for the rest of the paper. Recall that an ordered simplicial complex is a simplicial complex equipped with a total order on its vertices. Let $K$ be an ordered finite simplicial complex of dimension $n$ and let $g\colon K\to K$ be a simplicial map which preserves the order: if $v_i\le v_j$, then  $g(v_i)\le g(v_j)$.  
We will consider a subcomplex  $L$  of $K$ (with the induced order) such that $g(L)\subseteq L$ and we will consider a subcomplex $J$  of $L$ (with the induced order) such that $g(J)\subseteq J$. (Relative) Homology will always be considered with rational coefficients and $H_k(g)_{(L,J)}$  will denote the morphism $H_k(g)\colon H_k{(L,J)}\to H_k{(L,J)}$. We will denote by $\mathrm{Tr\,}H_k(g)$ and by $\mathrm{Im\,}H_k(g)$ the trace and the image, respectively,  of the homomorphism $H_k(g)$.

\begin{lemma}\label{lemma:key}
 For every integer $j\geq 0$:
 \begin{align*}
     \sum_{k=0}^j(-1)^{j-k}\mathrm{Tr\,}H_k(g)_{(K,J)}\leq& \sum_{k=0}^j(-1)^{j-k}\mathrm{Tr\,}H_k(g)_{(K,L)}\\ +& \sum_{k=0}^j(-1)^{j-k}\mathrm{Tr\,}H_k(g)_{(L,J)}
 \end{align*}
Moreover, the equality holds for $j=n$.
\end{lemma}

\begin{proof}
We begin by applying the naturality of the Long Exact Sequence in homology of the triple $(K,L,J)$ so the following diagram is commutative:
$$\begin{tikzcd}[column sep=small]
        \cdots \arrow[r] & H_j(L,J) \arrow[r,"\alpha_j"] \arrow[d, "H_j(g)_{(L,J)}"] & H_j(K,J) \arrow[r, "\beta_j"] \arrow[d, "H_j(g)_{(K,J)}"] & H_j(K, L) \arrow[r, "\eta_j"] \arrow[d, "H_j(g)_{(K,L)}"] & H_{j-1}(L,J) \arrow[r] \arrow[d, "H_{j-1}(g)_{(L,J)}"] & \cdots \\
        \cdots \arrow[r] & H_j(L,J) \arrow[r, "\alpha_j"] & H_j(K, J) \arrow[r, "\beta_j"] & H_j(K, L) \arrow[r, "\eta_j"] & H_{j-1}(L,J) \arrow[r] & \cdots 
    \end{tikzcd}$$

The exactness of the rows guarantees that the following diagram, which we will refer to as the brick for $H_j(g)_{(L,J)}$, is commutative:
\begin{equation*}
\begin{tikzcd}[column sep=small]
        0 \arrow[r] & \mathrm{Im\,} \eta_{j+1} \arrow[r] \arrow[d, "H_j(g)_{(L,J)_{|\mathrm{Im\,} \eta_{j+1}}}", swap]& H_j(L,J) \arrow[r,"\alpha_j"] \arrow[d, "H_j(g)_{(L,J)}"] & \mathrm{Im\,} \alpha_{j} \arrow[d, "H_j(g)_{(K,J)_{|\mathrm{Im\,} \alpha_{j}}}"] \arrow[r] & 0 \\
        0 \arrow[r] & \mathrm{Im\,} \eta_{j+1} \arrow[r] & H_j(L,J) \arrow[r, "\alpha_j"] & \mathrm{Im\,} \alpha_{j} \arrow[r] & 0 
    \end{tikzcd}   
\end{equation*}
    Therefore, 
    \begin{equation}\label{eq:1}
        \mathrm{Tr\,}H_j(g)_{(L,J)_{|\mathrm{Im\,} \eta_{j+1}}}=\mathrm{Tr\,}H_j(g)_{(L,J)}-\mathrm{Tr\,}H_j(g)_{(K,J)_{|\mathrm{Im\,} \alpha_{j}}}
    \end{equation}
From the exactness of the rows we also obtain that the following diagram (the brick for $H_j(g)_{(K,J)}$) is commutative:   
\begin{equation*}
\begin{tikzcd}[column sep=small]
        0 \arrow[r] & \mathrm{Im\,} \alpha_{j} \arrow[r] \arrow[d, "H_j(g)_{(K,J)_{|\mathrm{Im\,} \alpha_{j}}}", swap]& H_j(K,J) \arrow[r,"\beta_j"] \arrow[d, "H_j(g)_{(K,J)}"] & \mathrm{Im\,} \beta_j \arrow[d, "H_j(g)_{(K,L)_{|\mathrm{Im\,} \beta_j}}"] \arrow[r] & 0 \\
        0 \arrow[r] & \mathrm{Im\,} \alpha_{j} \arrow[r] & H_j(K,J) \arrow[r, "\beta_j"] & \mathrm{Im\,} \beta_j \arrow[r] & 0 
    \end{tikzcd}   
\end{equation*}

Hence,
\begin{equation}\label{eq:2}
        \mathrm{Tr\,}H_j(g)_{(K,J)_{|\mathrm{Im\,} \alpha_{j}}}=\mathrm{Tr\,}H_j(g)_{(K,J)}-\mathrm{Tr\,}H_j(g)_{(K,L)_{|\mathrm{Im\,} \beta_{j}}}.
    \end{equation}
The same reasoning applied to the brick for $H_j(g)_{(K,L)}$ gives:
\begin{equation}\label{eq:3}
        \mathrm{Tr\,}H_j(g)_{(K,L)_{|\mathrm{Im\,} \beta_{j}}}=\mathrm{Tr\,}H_j(g)_{(K,L)}-\mathrm{Tr\,}H_{j-1}(g)_{(L,J)_{|\mathrm{Im\,} \eta_{j}}}.
    \end{equation}
Combining Equations \eqref{eq:1} and \eqref{eq:2} we obtain:
\begin{align}\label{eq:4}
        \mathrm{Tr\,}H_j(g)_{(L,J)_{|\mathrm{Im\,} \eta_{j+1}}}=&\mathrm{Tr\,}H_j(g)_{(L,J)}-\mathrm{Tr\,}H_j(g)_{(K,J)} \nonumber\\ 
        &+\mathrm{Tr\,}H_j(g)_{(K,L)_{|\mathrm{Im\,} \beta_{j}}}.
    \end{align}
Combining Equations \eqref{eq:3} and \eqref{eq:4} we obtain:
\begin{align}\label{eq:5}
        \mathrm{Tr\,}H_j(g)_{(L,J)_{|\mathrm{Im\,} \eta_{j+1}}}=&\mathrm{Tr\,}H_j(g)_{(L,J)}-\mathrm{Tr\,}H_j(g)_{(K,J)} \nonumber\\ 
        &+\mathrm{Tr\,}H_j(g)_{(K,L)}-\mathrm{Tr\,}H_{j-1}(g)_{(L,J)_{|\mathrm{Im\,} \eta_{j}}}.
    \end{align}

Iterating this argument for the bricks for $H_k(g)_{(L,J)}$, $H_k(g)_{(K,J)}$ and  $H_k(g)_{(K,L)}$ for each $k\leq j$ and then combining the resulting equations we get:
\begin{align}\label{eq:6}
        \mathrm{Tr\,}H_j(g)_{(L,J)_{|\mathrm{Im\,} \eta_{j+1}}}=& \sum_{k=0}^j(-1)^{j-k}\mathrm{Tr\,}H_k(g)_{(K,L)}\nonumber\\ -&\sum_{k=0}^j(-1)^{j-k}\mathrm{Tr\,}H_k(g)_{(K,J)} \\ +& \sum_{k=0}^j(-1)^{j-k}\mathrm{Tr\,}H_k(g)_{(L,J)}.\nonumber
    \end{align}
    To see this, we proceed by induction. The case $j=0$ is Equation~\eqref{eq:5}. Suppose now that the statement is true for $j$. Equation~\eqref{eq:5} for $j+1$ implies
\begin{align*}
        \mathrm{Tr\,}H_{j+1}(g)_{(L,J)_{|\mathrm{Im\,} \eta_{j+2}}}=&\mathrm{Tr\,}H_{j+1}(g)_{(L,J)}-\mathrm{Tr\,}H_{j+1}(g)_{(K,J)} \nonumber\\ 
        &+\mathrm{Tr\,}H_{j+1}(g)_{(K,L)}-\mathrm{Tr\,}H_{j}(g)_{(L,J)_{|\mathrm{Im\,} \eta_{j+1}}}.
    \end{align*}
If we replace the term $\mathrm{Tr\,}H_{j}(g)_{(L,J)_{|\mathrm{Im\,} \eta_{j+1}}}$ by its value in Equation~\eqref{eq:6}, we get
\begin{align*}
        \mathrm{Tr\,}H_{j+1}(g)_{(L,J)_{|\mathrm{Im\,} \eta_{j+2}}}=&\mathrm{Tr\,}H_{j+1}(g)_{(L,J)}-\mathrm{Tr\,}H_{j+1}(g)_{(K,J)} \nonumber\\ 
        &+\mathrm{Tr\,}H_{j+1}(g)_{(K,L)}\\-&\bigg(\sum_{k=0}^j(-1)^{j-k}\mathrm{Tr\,}H_k(g)_{(K,L)}\\ -&\sum_{k=0}^j(-1)^{j-k}\mathrm{Tr\,}H_k(g)_{(K,J)} \\ + &\sum_{k=0}^j(-1)^{j-k}\mathrm{Tr\,}H_k(g)_{(L,J)}\bigg)\\=&\sum_{k=0}^{j+1}(-1)^{j+1-k}\mathrm{Tr\,}H_k(g)_{(K,L)}\\ -&\sum_{k=0}^{j+1}(-1)^{j+1-k}\mathrm{Tr\,}H_k(g)_{(K,J)} \\ + &\sum_{k=0}^{j+1}(-1)^{j+1-k}\mathrm{Tr\,}H_k(g)_{(L,J)}.
    \end{align*}
    
Now observe that it is enough to prove that $\mathrm{Tr\,}H_j(g)_{(L,J)_{|\mathrm{Im\,} \eta_{j+1}}}\geq 0$, since this implies the desired inequality as a consequence of Equation~\eqref{eq:6}. The statement about the equality follows from the fact that $\mathrm{Tr\,}H_n(g)_{(L,J)_{|\mathrm{Im\,} \eta_{n+1}}}= 0$.

So, it only remains to prove that $\mathrm{Tr\,}H_j(g)_{(L,J)_{|\mathrm{Im\,} \eta_{j+1}}}\geq 0$. Since $g:K\to K$ preserves the total order on the vertices and $K$ is finite, the sequence $\{g^r(v)\}_{r\ge 0}$
stabilizes. Hence, for each vertex $v$ there exists $N(v)\in\mathbb{N}$ such that
$g^{N(v)+1}(v)=g^{N(v)}(v)$. Let $N$ be the maximum of the $N(v)\in\mathbb{N}$. Then $g^{N+1}(v)=g^N(v)$ for all vertices.

Now let $\sigma=[v_0<\cdots<v_p]$ be an ordered simplex of $L$ (and hence of $K$).
Because $g$ preserves the order, the induced chain map sends $\sigma$ either to
the ordered simplex $[g(v_0)<\cdots<g(v_p)]$ (if the images are pairwise distinct)
or to $0$ (if the dimension drops). Since $g^{N+1}=g^N$ on vertices, we obtain that the induced maps on chain complexes satisfy $g_\#^{\,N+1}(\sigma)=g_\#^{\,N}(\sigma)$ for every simplex $\sigma\subset L$, and therefore
$g_\#^{\,N+1}=g_\#^{\,N}$ on the relative chain complex $C_*(L,J)$.

Passing to homology, we deduce $H_k(g)^{N+1}=H_k(g)^{N}$ on $H_k(L,J)$ for all $k$. By commutativity of the long exact sequence, we have that  $H_j(g)_{(L,J)}(\mathrm{Im\,} \eta_{j+1})\subset \mathrm{Im\,} \eta_{j+1}$, hence the restriction
\[
T:=H_j(g)_{(L,J)_{|\mathrm{Im\,} \eta_{j+1}}}
\]
also satisfies $T^{N+1}=T^{N}$. Let $\lambda\in \mathbb{C}$ be an eigenvalue of $T$ and let $v\neq 0$ be a corresponding eigenvector. Then $T^{N}v=\lambda^{N}v$ and $T^{N+1}v=\lambda^{N+1}v$. Using $T^{N+1}=T^{N}$ we obtain $\lambda^{N+1}v=\lambda^{N}v$, hence $\lambda^{N}(\lambda-1)\,v=0$. Since $v\neq 0$, it follows that $\lambda^{N}(\lambda-1)=0$, and therefore $\lambda\in\{0,1\}$. 
Consequently,
\[
\mathrm{Tr\,}\Bigl(H_j(g)_{(L,J)_{|\mathrm{Im\,} \eta_{j+1}}}\Bigr)=\mathrm{Tr\,}(T)\ge 0,
\]
which is what remained to be proved.
\end{proof}

\subsection{Forman-Morse-Bott functions}
We introduce a special kind of function to encode the Morse-Bott dynamics on $K$. In order to do so, we fix some notation. For a subset $C$ of the simplices of $K$, $\overline{C}$ denotes the subcomplex generated by $C$, that is the subcomplex consisting of all simplices of $C$ and all their faces.

Let $f$ be a map from the set of simplices of $K$ to ${\{0,\dots,m\}}$ for some $m\in \mathbb{N}$.  For each $i\in\{0,\dots,m\}$ we define $K_i=\overline{\{\sigma\in K\, \colon f(\sigma)\leq i\}}$.

\begin{defi}\label{defi:Forman-Morse-Bott}
    A Forman-Morse-Bott function is a map $f$ from the set of simplices of $K$ to $\{0,1,\ldots,m\}$  (which we will write abusing notation as $f\colon K\to \{0,\dots, m\}$) for some $m\in \mathbb{N}$ satisfying the following conditions. For each $i\in \{0,1,\ldots,m-1\}$, there exists $p\in \mathbb{N}$ such that:\begin{enumerate}
        \item $H_k(K_{i+1},K_i)= 0$ for $k\neq p,p+1$.
        \item $K_{i+1}-K_i$ consists of three mutually exclusive possibilities:         
        \begin{enumerate}
            \item[(2a)] a  unique simplex $\sigma$ of dimension $p$ satisfying $H_{p+1}(K_{i+1},K_i)= 0$ and $H_p(K_{i+1},K_i)\cong \mathbb{Q}$.
            \item[(2b)] a set of simplices of dimensions $p$ and $p+1$ with at least one of dimension $p$ and at least one of dimension $p+1$. Moreover, exactly one of following three conditions hold: \begin{enumerate}
            \item $H_p(K_{i+1},K_i)= 0$ and  $H_{p+1}(K_{i+1},K_i)\cong \mathbb{Q}$.
            \item $H_{p}(K_{i+1},K_i)= \mathbb{Q}$ and $H_{p+1}(K_{i+1},K_i)\cong 0$. 
            \item $H_p(K_{i+1},K_i)= \mathbb{Q}$ and  $H_{p+1}(K_{i+1},K_i)\cong \mathbb{Q}$.
        \end{enumerate}  
            \item[(2c)] a pair of simplices: $\sigma$ of dimension $p$ and $\tau$ of dimension $p+1$ satisfying $H_{p+1}(K_{i+1},K_i)\cong H_{p}(K_{i+1},K_i) = 0$.
        \end{enumerate}  
    \end{enumerate}
\end{defi}

Recall the notion of a Morse-Smale combinatorial vector field on a simplicial complex $K$ from \cite{Forman_Bott}. Intuitively, it is a vector field on the simplices of $K$ whose critical set consists of isolated simplices or simple periodic orbits. Following the ideas of \cite[Theorem 3.11, Definition 3.13, Theorem 5.5]{DML}, it can be seen that a Liapunov or discrete Morse-Bott function (see \cite{Forman_Bott} for a definition) obtained by integrating a Morse-Smale combinatorial vector field on $K$ can be perturbed to be a Forman-Morse-Bott function on $K$. That is to say, one way to obtain a Forman-Morse-Bott function on $K$ is by integrating a Morse-Smale combinatorial vector field on $K$ (in the sense of \cite{Forman_Bott}).

\begin{remark}
    We provide some dynamical intuition for Definition \ref{defi:Forman-Morse-Bott}. The case (2a) models a critical point (simplex), the case (2b) models a critical periodic orbit and the (2c) models the gradient part of the vector field.
\end{remark}

\subsection{The theorem} We begin by introducing a measure of the contribution to the Lefschetz number at critical objects (that is, critical simplices and periodic orbits).

\begin{defi}
We define the local $k$-trace of $g\colon K\to K$ as
    $$\text{Local $k$-trace of $g$}=\sum_{i\geq -1}\mathrm{Tr} \, H_k(g)_{(K_{i+1},K_{i})}$$
    where, by convention, $K_{-1}=\emptyset$.
\end{defi}

Let $l\geq 1$. We denote $g^l=\underbrace{g \circ {\cdots} \circ g}_{\text{$l$ times}}$. 

\begin{teo}[Strong and Weak inequalities]\label{thm:teo}
    Let $K$ be a finite ordered simplicial complex of dimension $n$ and let $g\colon K \to K$ be a simplicial map which preserves the total order of $K$. Let $f\colon K\to \{0,\dots,m\}$ be a Forman-Morse-Bott function that satisfies $f(g(\sigma))\leq f(\sigma)$ for every simplex $\sigma$ of $K$. For every  $l\geq 1$ and for every $0\leq j\leq n$  it holds that:
    
    \begin{equation}\label{eq:strong}
        \sum_{k=0}^{j}(-1)^{j-k}\mathrm{Tr} \, H_k(g^l) \leq \sum_{k=0}^{j}(-1)^{j-k}{\textnormal{Local $k$-trace of $g^l$}}
    \end{equation}
    and also the equality holds for $j=n.$
    Moreover, for every $l\geq 1$ and for every $0\leq k\leq n$:
    \begin{equation}\label{eq:weak}
        \mathrm{Tr} \, H_k(g^l) \leq \textnormal{Local $k$-trace of $g^l$}
    \end{equation}
\end{teo}

\begin{proof}
First of all, since, for all $l\geq 1$, $$f(g^{l+1}(\sigma))\leq f(g^{l}(\sigma)) \leq \cdots \leq f(g(\sigma))\leq f(\sigma)$$ for every simplex $\sigma$ of $K$, it is enough to prove the results for $l=1$.
We begin by proving Equation \eqref{eq:strong}. Consider the filtration:

$$\emptyset=K_{-1}\subseteq K_0 \subseteq K_1 \subseteq \cdots \subseteq K_i \subseteq K_{i+1} \subseteq \cdots \subseteq K_m=K$$

of $K$ induced by the Forman-Morse-Bott function 
$f\colon K\to \{0,\dots,m\}$. 
We apply Lemma \ref{lemma:key} inductively on the filtration:
 \begin{align*}
     \sum_{k=0}^j(-1)^{j-k}\mathrm{Tr\,}H_k(g)_{(K_m,K_{-1})}\leq& \sum_{k=0}^j(-1)^{j-k}\mathrm{Tr\,}H_k(g)_{(K_{m-1},K_{-1})}\\
     +& \sum_{k=0}^j(-1)^{j-k}\mathrm{Tr\,}H_k(g)_{(K_{m},K_{m-1})}\\
     \leq&  \sum_{k=0}^j(-1)^{j-k}\mathrm{Tr\,}H_k(g)_{(K_{m-2},K_{-1})}\\
     +& \sum_{k=0}^j(-1)^{j-k}\mathrm{Tr\,}H_k(g)_{(K_{m-1},K_{m-2})}\\
     +& \sum_{k=0}^j(-1)^{j-k}\mathrm{Tr\,}H_k(g)_{(K_{m},K_{m-1})}\\
     \leq& \cdots\\
     \leq& \sum_{k=0}^{j}(-1)^{j-k}\text{Local $k$-trace of $g^l$}
 \end{align*}

To end the proof of Equation \eqref{eq:strong}, observe that 
$$\sum_{k=0}^j(-1)^{j-k}\mathrm{Tr\,}H_k(g)_{(K_m,K_{-1})}=\sum_{k=0}^{j}(-1)^{j-k}\mathrm{Tr} \, H_k(g^l)$$ and that Lemma \ref{lemma:key} guarantees the equality for $j=n$. 
Equation \eqref{eq:weak} is deduced from adding two consecutives values for $j$ in Equation \eqref{eq:strong}. 
\end{proof}

We refer to Equation \eqref{eq:strong} as the Strong Morse-Bott inequalities for endomorphisms and to  Equation \eqref{eq:weak} as the Weak Morse-Bott inequalities for endomorphisms. 

\begin{remark}
    Observe that the the equality of Equation {\eqref{eq:strong}} for $j=n$ provides a localization of the Lefschetz number at the critical objects (critical simplices and periodic orbits) of the combinatorial vector field.
\end{remark}

\begin{remark}
    The compatibility condition between the Morse-Bott dynamics and the map $g$, $f(g(\sigma))\leq f(\sigma)$ for every simplex $\sigma$ of $K$, means that $g$ respects the direction of the flowlines of the Morse-Bott dynamics in $K$. 
\end{remark}

\begin{remark}
    Taking $g\colon K \to K$ to be the identity and $f$ a Liapunov function obtained from integrating a Morse-Smale combinatorial vector field we obtain as a particular case the Strong and Weak Morse-Bott inequalities (\cite[Theorem 3.1]{Forman_Bott}).
\end{remark}

\begin{example}
Consider the simplicial complex $K$ depicted in Figure \ref{fig:complexo} with the order: $v_0\leq v_1\leq v_2\leq v_5\leq v_4\leq v_3$, the simplicial map $g\colon K\to K$ given by: 
\begin{align*}
    & g(v_0)=v_0, \quad g(v_1)=v_1, \quad g(v_2)=v_2,\\
    & g(v_3)=v_2, \quad g(v_4)=v_1, \quad g(v_5)=v_0;
\end{align*}%$g(v_0)=v_0$, $g(v_1)=v_1$, $g(v_2)=v_2$. $g(v_3)=v_2$, $g(v_4)=v_1$ and $g(v_5)=v_0$, 
and the Forman-Morse-Bott function $f\colon K\to \{0,1,2,3,4,5,6\}$ given by:
\begin{align*}
    &f(v_0)=1, \quad &f(v_1)=2, \qquad &f(v_2)=0, \qquad &f([v_0,v_1])=3,\\
    &f(v_4)=4, \quad &f(v_5)=4, \qquad &f(v_3)=4, \qquad &f([v_0,v_2])=1,\\
    &f([v_1,v_2])=2,  &f([v_2,v_3])=5, \qquad &f([v_3,v_4])=4, \qquad &f([v_4,v_5])=3,\\
    &f([v_3,v_5])=3, \quad & & &f([v_0,v_1,v_2])=6.
\end{align*}

%$f(v_0)=1$, $f(v_1)=2$, $f(v_2)=0$, $f(v_3)=4$, $f(v_4)=4$, $f(v_5)=4$, 

We exhibit the computations of the local \(k\)-traces of \(g\) for each \(k\). For the sake of clarity of the presentation, we omit those terms in the filtration for which \(\mathrm{Tr\,} H_k(g)(K_{i+1},K_i)=0\).
\[
\begin{aligned}
\text{local $0$-trace of } g
&= \mathrm{Tr}\, H_0(g)_{(K_0,\emptyset)}=1.
\end{aligned}
\]
\[
\begin{aligned}
\text{local $1$-trace of } g
&= \mathrm{Tr}\, H_1(g)_{(K_3,K_2)}=1.
\end{aligned}
\]
\[
\begin{aligned}
\text{local $2$-trace of } g
&= \mathrm{Tr}\, H_2(g)_{(K_6,K_5)}=1.
\end{aligned}
\]
\begin{figure}[h!]
    \centering
\includegraphics[width=0.75\linewidth]{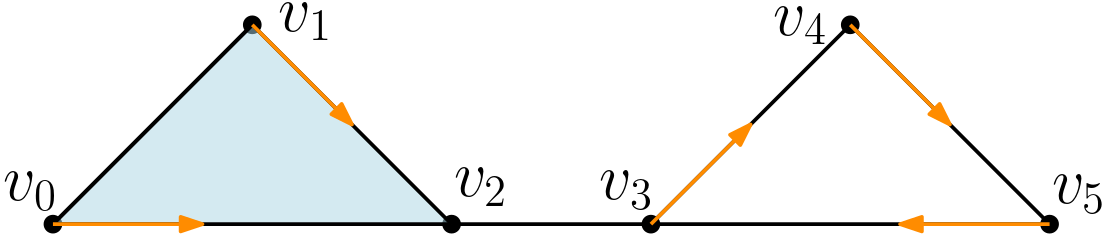}
    \caption{Simplicial complex $K$ and combinatorial Morse-Smale combinatorial vector field.}
    \label{fig:complexo}
\end{figure}
We introduce two tables comparing the Strong (respectively Weak) Morse-Bott inequalities with the Strong (respectively Weak) Morse-Bott inequalities for endomorphisms.
    \begin{table}[h!]
    \centering
    \begin{tabular}{|c|c|c|}
        \hline
        \multicolumn{3}{|c|}{\textbf{Strong Inequalities}} \\
        \hline
        & Morse-Bott & Morse-Bott for the endomorphism $g$ \\
        \hline
        $j=0$ & $1 \leq 2$ & $1 \leq 1$ \\
        $j=1$ & $0 \leq 1$ & $-1 \leq 0$ \\
        $j=2$ & $0 \leq 0$ & $1 \leq 1$ \\
        \hline
    \end{tabular}
    \caption{Strong Morse-Bott inequalities for the endomorphism $g$.}
\end{table}

\begin{table}[h!]
    \centering
    \begin{tabular}{|c|c|c|}
        \hline
        \multicolumn{3}{|c|}{\textbf{Weak Inequalities}} \\
        \hline
        & Morse-Bott & Morse-Bott for the endomorphism $g$\\
        \hline
        $j=0$ & $1 \leq 2$ & $1 \leq 1$ \\
        $j=1$ & $1 \leq 3$ & $0\leq 1$ \\
        $j=2$ & $0 \leq 1$ & $0 \leq 1$ \\
        \hline
    \end{tabular}
    \caption{Weak Morse-Bott inequalities for the endomorphism $g$.}
\end{table}

\end{example}

%\bibstyle{sn-mathphys}
\bibliographystyle{plain}
\bibliography{biblio}% common bib file HACER ESTE FICHERO

%\input sn-article.bbl


\end{document}